\documentclass[12pt]{article}

\usepackage[margin=1in]{geometry}

\usepackage{amsmath, amssymb, amsthm, amssymb}
\usepackage{amsfonts}
\usepackage{graphicx}
\usepackage[frame,ps,matrix,arrow,curve,rotate]{xy}
\usepackage{enumerate}
\usepackage{hyperref}
\usepackage{siunitx}

\usepackage{tikz-cd}

\usepackage{blkarray}

\newtheorem{theorem}{Theorem}[section]
\newtheorem{thm}[theorem]{Theorem}
\newtheorem{lemma}[theorem]{Lemma}

\newtheorem{corollary}[theorem]{Corollary}

\newcounter{claims}[theorem]
\newtheorem{claim}[claims]{Claim}

\theoremstyle{definition}

\newtheorem{defn}[theorem]{Definition}
\newtheorem{example}[theorem]{Example}

\newtheorem{question}[theorem]{Question}

\theoremstyle{remark}

\newcommand{\mc}[1]{\mathcal{#1}}

\newcommand{\la}{\langle}
\newcommand{\ra}{\rangle}

\DeclareMathOperator{\ran}{ran}

\DeclareMathOperator{\dom}{dom}

\makeatletter
\newcommand\mathcircled[1]{%
  \mathpalette\@mathcircled{#1}%
}
\newcommand\@mathcircled[2]{%
  \tikz[baseline=(math.base)] \node[draw,circle,inner sep=1pt] (math) {$\m@th#1#2$};%
}
\makeatother

\begin{document}

\title{Two results on complexities of decision problems of groups}
\author{Uri Andrews\thanks{The authors would like to thank the CSU Desert Studies Center for supporting this project.}, Matthew Harrison-Trainor\footnotemark[1] \thanks{The second author acknowledges support from the National Science Foundation under Grant No.\ \mbox{DMS-2153823}.}, Turbo Ho\footnotemark[1] \thanks{The third author acknowledges support from the National Science Foundation under Grant No.\ \mbox{DMS-2054558}.}}

\maketitle

\begin{abstract}
We answer two questions on the complexities of decision problems of groups, each related to a classical result.  First, C.\ Miller characterized the complexity of the isomorphism problem for finitely presented groups in 1971. We do the same for the isomorphism problem for recursively presented groups.  Second, the fact that every Turing degree appears as the degree of the word problem of a finitely presented group is shown independently by multiple people in the 1960s. We answer the analogous question for degrees of ceers instead of Turing degrees. We show that the set of ceers which are computably equivalent to a finitely presented group is $\Sigma^0_3$-complete,  which is the maximal possible complexity.
\end{abstract}

\section{Introduction}

The most famous decision problems in group theory are the three problems proposed by Dehn \cite{De11}: the word problem, the conjugacy problem, and the isomorphism problem. These all turned out to be incomputable. In this paper, we focus on the first and third problems. Novikov \cite{No55}, and independently Boone \cite{Bo57}, showed that there are finitely presented groups with unsolvable word problem. More generally, for any c.e.\ degree there is a finitely presented group whose word problem is of that degree \cite{Fr62, Cl64, Bo66, Bo67}. Adyan \cite{Ad55} showed that the isomorphism problem is incomputable. 

Though much work on decision problems have used the Turing degrees as the ultimate measure of complexity, we will examine a more refined measure.
We consider these decision problems as equivalence relations and we use computable reduction as a way to compare complexity of equivalence relations. If $E$ and $F$ are equivalence relations on $\omega$, we say $E$ \emph{computably reduces to $F$, denoted $E\leq F$,} if there is a total computable function $\Phi$ so that $i\mathrel{E} j$ if and only if $\Phi(i)\mathrel{F} \Phi(j)$.

Decision problems in group theory are generally considered on three strata of groups, based on the complexity of their presentations. The \emph{recursively presented} groups are those of the form $\langle X\mid R\rangle $ where $X$ is a computable set of generators and $R$ is a computable (equivalently c.e.) set of relators. The \emph{finitely generated} recursively presented groups require $X$ to also be finite. The most restrictive class is the \emph{finitely presented} groups, which require both $X$ and $R$ to be finite.

\subsection{The Isomorphism Problem}

Our first main result is on the isomorphism problem for finitely generated groups. C.\ Miller 
 \cite{CharlesFMillerIII} showed that the isomorphism problem for finitely presented groups is as hard as possible. In particular, he showed that the isomorphism problem for finitely presented groups is a complete c.e.\ equivalence relation (ceer).\footnote{Miller did not use this terminology as this was before ceers were a topic of study in their own right.} That is, it is a $\Sigma^0_1$ equivalence relation and every other $\Sigma^0_1$-equivalence relation computably reduces to it. 

In this paper, we show that the isomorphism problem for recursively presented or finitely generated recursively presented groups are as hard as possible, namely they are $\Sigma^0_3$-complete equivalence relations.
In fact, we show the following stronger result:

\begin{thm}
	The isomorphism problem for $6$-generated recursively presented groups is a $\Sigma^0_3$-complete equivalence relation, i.e., every other $\Sigma^0_3$ equivalence relation computably reduces to it.
\end{thm}


\subsection{The Word Problem}

Our second main result is on the word problem. Recall that for any c.e.\ Turing degree, there is a finitely presented group whose word problem is of that Turing degree. However, this does not hold in the finer world of equivalence relations: not every c.e.\ equivalence relation degree (under computable reduction) contains the word problem of a recursively presented group. This can be seen from the fact that any equivalence relation which is the word problem of a group has the same Turing degree as each of its equivalence classes. There are non-computable ceers where every equivalence class is computable. Then any ceer in its degree will also be non-computable but have computable equivalence classes. Such degrees cannot contain recursively presented groups. 

Further, Della Rose, San Mauro, and Sorbi showed that there are degrees that contain word problems of recursively presented groups but not of finitely generated recursively presented groups \cite{De23}.\footnote{In \cite{De23}, they introduce the notion of a hyperdark degree, which cannot be realized as the word problem of a finitely generated algebra. They also construct a semigroup with a hyperdark word problem, and with slight modification, this construction can be made to yield a group.}

On the other hand, we do not know if there are degrees containing finitely generated recursively presented groups but not finitely presented groups. Given this situation, it is sensible to ask the complexity of the collection of degrees which contain word problems of finitely presented, or finitely generated recursively presented, or recursively presented, groups. To formalize the question, we use the natural enumeration of all ceers: $E_i$ for $i\in \omega$. Then we consider the set of $i$ so that $E_i$ is computably equivalent to a word problem of a finitely presented group, or similarly for the other classes of groups.

We prove a general result which answers all of these questions simultaneously.
\begin{defn}
	For $S\subset \omega$, which we think of as a set of indices of ceers, let $[S]=\{i\mid \exists j\in S (E_i\equiv E_j)\}$ be the closure of $S$ under computable equivalence.
	
	We say $S$ is \emph{proper on the infinite ceers} if there are $E_i$ and $E_j$ both with infinitely many classes so that $i\in [S]$ and $j\notin [S]$.
\end{defn}

We prove the following theorem, which we think of as an analogue of Rice's Theorem for index sets of ceers.

\begin{thm}
	If $S$ is $\Sigma^0_3$ and proper on the infinite ceers, then $[S]$ is $\Sigma^0_3$-complete.
\end{thm}

Using this theorem we see that the collection of $i$ such that $E_i$ is equivalent to the word problem of a finitely presented group is $\Sigma^0_3$-complete, and similarly for the other two classes of groups. We also show that the assumption that $S$ is $\Sigma^0_3$-complete is optimal by constructing a $\Pi^0_3$ set which is proper on the infinite ceers but $[T]$ is neither $\Sigma^0_3$ nor $\Pi^0_3$-hard.

\section{The isomorphism problem for finitely generated groups}

In this section, we will show that the isomorphism problem for finitely generated recursively presented groups is $\Sigma^0_3$-complete. We use two pieces of heavy machinery. Firstly, we use the $\Sigma^0_3$-completeness of a particular convenient equivalence relation, which follows from a general criterion for $\Sigma^0_3$-completeness proved by Andrews and San Mauro \cite{ASM1}.

\begin{defn}
	Let $(V_i)_{i \in \omega}$ be a uniform enumeration of all of the c.e.\ subsets of $\mathbb{Z}$. For $X \subseteq \mathbb{Z}$ and $n\in \omega$, we define $X+n$ as the set $\{m+n\mid m\in X\}$. We define the equivalence relation $U$ given by the shift operator $L$ on c.e.\ sets as follows:
	$$i\mathrel{U} j \Leftrightarrow \exists n \left[ V_i = (V_j + n) \right].$$
\end{defn}

\begin{thm}[An application of Andrews and San Mauro \cite{ASM1}] \label{UisUniversal}
	
	The equivalence relation $U$ is a $\Sigma^0_3$-complete equivalence relation.
\end{thm}

From this, we show that there is a signature $\mathcal{L}$ so that the isomorphism problem for 1-generated recursively presented $\mathcal{L}$-structures is $\Sigma^0_3$-complete:

\begin{lemma}\label{lem:univ-unary-str}
	Let $\mathcal{L}=\{f,g\}\cup \{h_i\mid i\in \omega\}$ where each symbol is a unary function symbol. Then the isomorphism problem for $1$-generated recursively presented $\mathcal{L}$-structures is $\Sigma^0_3$-complete.
\end{lemma}
\begin{proof}
First observe that the isomorphism problem, namely, saying there is an isomorphism between $1$-generated $\mc A$ to $\mc B$ as given by recursive presentations, is $\Sigma^0_3$. For $\Sigma^0_3$-hardness, it suffices by Theorem \ref{UisUniversal} to show that $U$ computably reduces to the isomorphism problem for $1$-generated recursively presented $\mathcal{L}$-structures.
	
For a c.e.\ set $V_r\subseteq \mathbb{Z}$, we effectively produce a c.e.\ presentation of a 1-generated $\mathcal{L}$-structure $\mathcal{A}_r$ such that $r \mathrel{U} s$ if and only if $\mathcal{A}_r \cong \mathcal{A}_s$. For a given $r$, the construction of $\mc{A}_r$ is as follows.
	
Let $\mathcal{A}_r$ have universe $\{x_i\mid i\in \mathbb{Z}\}\cup \{y_{i,j}\mid i,j\in \mathbb{Z}\}\cup \{z_{i,j}^k\mid i,j\in \mathbb{Z}, k\in \omega, j-i\notin V_r\}$. In $\mc{A}_r$, let $f(x_i)=x_{i+1}$ and $f(v)=v$ if $v$ is not an $x_i$. Let $g=f^{-1}$. Let $h_j(x_i)=y_{i,j}$, $h_j(y_{i,k}) = y_{i,k}$ if $j\neq k$. Let $h_{j}(y_{i,j}) = y_{i,j}$ if $j-i\in V_r$, and let $h_{j}(y_{i,j}) = z_{i,j}^0$ if $j-i\notin V_r$. Finally, $h_j(z^k_{i,j'}) = z^k_{i,j'}$ if $j\neq j'$ and $h_j(z^k_{i,j}) = z^{k+1}_{i,j}$. (See Figure \ref{fig}.)
\begin{figure}
\caption{An example where $j-i \notin V_r$ but $j'-i\in V_r$.}\label{fig}
\begin{tikzcd}
\cdots \arrow[rr,"f"]	&	&x_{i-1} \arrow[rr,"f"] 	&							&x_{i} \arrow[rr,"f"] \arrow[ld,"h_j"]\arrow[rd,"h_{j'}"]	&									&x_{i+1} \arrow[rr,"f"]		&	&\cdots 	\\
					&	&\cdots					&y_{i,j}\arrow[d,"h_j"]		&\cdots													&y_{i,j'}\arrow[loop below,"h_{j'}"]	&\cdots						&	&		\\
					&	&						&z^0_{i,j}\arrow[d,"h_j"]	&														&									&							&	&		\\
					&	&						&z^1_{i,j}\arrow[d,"h_j"]	&														&									&							&	&		\\
					&	&						&\vdots						&														&									&							&	&		
\end{tikzcd}
\end{figure}
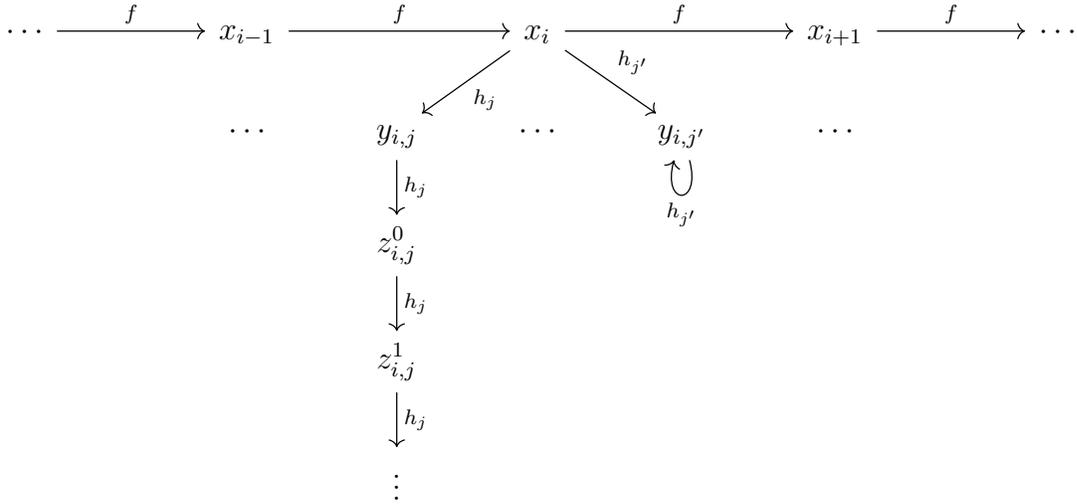

	First, observe that $\mathcal{A}_r$ is generated by $x_0$: Using $f$ and $g$, $x_0$ generates $\{x_i\mid i\in \mathbb{Z}\}$, then the $h_j$-functions generate all the $y_{i,j}$'s. Finally, iteratively applying the $h_j$-functions to the $y_{i,j}$ where $j-i\notin V_r$ generates all of the $z_{i,j}^k$.
	
	Next observe that $\mathcal{A}_r$ can be given as a recursive presentation $\langle x_0\mid S\rangle$ where $S$ is a c.e.\ collection of relators in $x_0$. These relators need to say that $f$ and $g$ are inverses for every element, that $f$ is the identity on elements in the range of each $h_i$. Similarly, each $h_j$ is the identity on any element in the range of $h_i$ for $i\neq j$. Finally, we need an axiom saying that $h_j(h_j(f^i(x_0)) = h_j(f^i(x_0))$ if $i-j\in V_r$, and $h_j(h_j(g^i(x_0)) = h_j(g^i(x_0))$ if $-i-j\in V_r$. This is a c.e.\ list of axioms, and can be turned into a computable list of axioms by putting $f^{s} (h_j(h_j(g^i(x_0)) ) = h_j(g^i(x_0))$ if $s$ is the stage at which we see $i-j$ enter $V_r$.
	
	Next observe that if $\alpha$ is an isomorphism of $\mc{A}_r$ to $\mc{A}_s$, then it is determined by $\alpha(x_0^{\mc{A}_r})$, which is necessarily the element $x_m^{\mc{A}_s}$ for some $m$ (since it must be an element moved by $f$). It is then straightforward to see that any homomorphism with $\alpha(x_0^{\mc{A}_r}) = x_m^{\mc{A}_s}$ gives an isomorphism if and only if the truth values of $h_j(y_{i,j}) = y_{i,j}$ are preserved for each $i,j$. In particular, we must have $j-i\notin V_r$ if and only if $j-i-m\notin V_s$. That is, $V_r = V_s + m$.
	
	On the other hand, if $V_r = V_s + m$ then $x_0^{\mc{A}_r} \mapsto x_m^{\mc{A}_s}$ induces an isomorphism of $\mc{A}_r$ to $\mc{A}_s$.	Thus $\mathcal{A}_r\cong \mathcal{A}_s$ if and only if $r \mathrel{U} s$.
\end{proof}

To transfer the completeness to the isomorphism problem for groups, we use the machinery from Harrison-Trainor and Ho \cite{Ha21} where it was shown that finitely generated groups are universal under reductions via infinitary bi-interpretability. The idea is to use small cancellation theory to encode an arbitrary finitely generated structure---such as $\mc{A}_r$---into a finitely generated group, where the functions of $\mc{A}_r$ become words in the group.

\begin{theorem}
	The isomorphism problem for $6$-generated recursively presented groups is $\Sigma^0_3$-complete.
\end{theorem}
\begin{proof}

First observe that the isomorphism problem, namely, saying there is an isomorphism between $6$-generated groups $\mc A$ and $\mc B$ as given by recursive presentations, is $\Sigma^0_3$. 

For $\Sigma^0_3$-hardness, by Lemma \ref{lem:univ-unary-str} it suffices to reduce the isomorphism problem for $1$-generated recursively presented $\mc L$-structures to the isomorphism problem for $6$-generated recursively presented groups. Namely, we need an effective procedure to enumerate the presentation of a $6$-generated group $G(\mc A)$ from any c.e.\ presentation of a $1$-generated $\mc L$-structure $\mc A$ so that $\mc A \cong \mc B$ if and only if $G(\mc A) \cong G(\mc B)$. 

Fix $\mathcal{L}=\{f,g\}\cup \{h_i\mid i\in \omega\}$ as in Lemma \ref{lem:univ-unary-str}. The construction in \cite[\S 3.3]{Ha21} gives, for every $\mathcal{L}$-structure $\mathcal{A}$, a group $G(\mc A)$ such that $\mc{A} \cong \mc{B}$ if and only if $G(\mc{A}) \cong G(\mc{B})$. In \cite{Ha21} the focus was on computable copies of the structures, and it was verified that given a copy of $\mc{A}$ we can compute a copy of $G(\mc{B})$. For our purposes in this paper we must verify that this construction is effective on presentations of the structures, that is, given a presentation of $\mc{A}$ we can compute a presentation of $G(\mc{A})$; we prove this in Claim \ref{claim:one}. (The difference between a computable copy and a presentation is that a computable copy includes a solution of the word problem.) We repeat the construction here, slightly simplifying it as we only have unary functions.

We denote the identity of the group as $e$. Define the presentation ${\la \{a\}_{a \in A},b,c,d,f_1,f_2 \mid P \ra}$ of $G(\mc A)$ to have generators $\{a\}_{a\in A} \cup \{b,c,d,f_1,f_2\}$ and the following set $P$ of relators:
\begin{itemize}
	\item $f_i^{p_i} = e$, where $p_1,p_2$ are distinct primes greater than $10^{10}$,
	\item $v(f_1,f_2) = v(f_2,f_1) = e$,
	\item $u_b(b,f_1) = u_b(b,f_2) = e$, $u_c(c,f_1) = u_c(c,f_2) = e$, $u_d(d,f_1) = u_d(d,f_2) = e$,
	\item $u_A(a,c) = u_A(a,d) = e$ for $a \in A$,
	\item $w_f(a,b) = a'$ if $f(a) = a'$, $w_g(a,b) = a'$ if $g(a) = a'$, and $w_{h_i}(a,b) = a'$ if $h_i(a) = a'$ for $a,a'\in A$.
\end{itemize}
where
\begin{itemize}
	\item $u_A(x,y) = x y x y^{4} \cdots x y^{999^2} x y^{1000^2}$,
	\item $u_b(x,y) = x y^{1001^2} x y^{1002^2} \cdots x y^{1999^2} x y^{2000^2}$,
	\item $u_c(x,y) = x y^{2001^2} x y^{2002^2} \cdots x y^{2999^2} x y^{3000^2}$,
	\item $u_d(x,y) = x y^{3001^2} x y^{3002^2} \cdots x y^{3999^2} x y^{4000^2}$,
	\item $v(x,y) = x y^{4001^2} x y^{4002^2} \cdots x y^{4999^2} x y^{5000^2}$,
	\item $w_f(x,y) = x y^{100+1} x y^{100 + 2} \cdots x y^{100 + 100}$, 
	\item $w_g(x,y) = x y^{200+1} x y^{200 + 2} \cdots x y^{200 + 100}$, 
	\item $w_{h_i}(x,y) = x y^{100(i+3)+1} x y^{100(i+3) + 2} \cdots x y^{100(i+3) + 100}$.

\end{itemize}

In the above presentation, we use the elements of $\mc{A}$ as generators for $G(\mc{A})$. But if $\mc{A} \cong \langle X \mid S \rangle$ is only a c.e.\ presentation, then in this c.e.\ presentation we do not have a listing of the elements of $\mc{A}$ since we do not know which terms are equal. To obtain a c.e.\ presentation of $G(\mc{A})$ from a c.e.\ presentation of $\mc{A}$, we must slightly modify the above presentation of $G(\mc{A})$.

\begin{claim}\label{claim:one}
Given a c.e.\ presentation of $\mc A$, we can find a c.e.\ presentation of $G(\mc A)$. This is uniform.
\end{claim}
\begin{proof}
Let $\langle X \mid S \rangle$ be a c.e.\ presentation of $\mc A$. We will enumerate a presentation $P'$ of $G(\mc A)$ which contains more generators and slightly different relations than $P$. 

The generating set of our presentation will be $T\cup \{b,c,d,f_1,f_2\}$, where $T$ is the collection of all the $\mc L$-terms in $X$. We do this by enumerating all $\mc L$-terms as we enumerate $X$. In addition to the relations in $P$, we also enumerate $w_\gamma(a,b) = \gamma(a)$ for every $\gamma \in \mc L$ and $a\in T$. 

To enumerate the relations in $P$, first observe that the first three kinds of relations do not depend on $\mc A$ at all, so we simply enumerate them at the start. We enumerate relations of the fourth kind as we enumerate the generators $\{a\}_{a\in \mc A}$. Lastly, at every stage, we enumerate all the consequences of $S$. As we enumerate a consequence $t = t'$ of $S$ (where $t,t' \in T$ are terms in the original presentation $P$), we enumerate the relation $t = t'$ (where $t, t'$ are generators in the new presentations $P'$) into $P'$.

We see that the set $T$, modulo equality in $P'$, is in a natural bijection with the domain of $\mc A$. Indeed, $T$ naturally surjects onto $\mc A$, and if two terms are equal in $\mc A$, it must be a consequence of some finitely many relations in $S$, so we will also enumerate the identity into $P'$. This bijection from $T$ to $\mc A$ induces a bijection from $P'$ to $P$, which is also a homomorphism. So the c.e.\ presentations $P'$ and $P$ define isomorphic groups, namely $G(\mc A)$. 
\end{proof}

The fact that $G(\mc{A})$ actually defines a group $G(\mc{A})$ encoding the structure of $\mc{A}$ is the same as in \cite{Ha21}. We refer the reader there for the argument, but here briefly summarize that the relations $u_A(a,c) = u_A(a,d) = e$ serve to define the elements of $A$. The relations $w_f(a,b) = a'$ if $f(a) = a'$ serve to define $f$ inside of $G(\mc{A})$, and similarly for $g$ and the $h_i$. Finally, small cancellation theory plays the role of ensuring that there is not too much collapse, e.g., that there are no unintended consequences of the relators.

\begin{claim}
	Two $\mc{L}$-structures $\mc{A}$ and $\mc{B}$ are isomorphic if and only if $G(\mc{A})$ is isomorphic to $G(\mc{B})$.
\end{claim}

\begin{proof}
This follows from \cite[Theorem 3.1]{Ha21} and \cite[Lemma 3.9]{Ha21}, which says that the orbit of $(b,c,d,f_1,f_2)$ is definable. 
\end{proof}

\begin{claim}
If $\mc A$ is $k$-generated, then $G(\mc A)$ is $k+5$-generated.
\end{claim}

\begin{proof}
This follows from the proof of \cite[Lemma 3.11]{Ha21}.
\end{proof}

These three claims complete the proof of the theorem.
\end{proof}

It is not known if the number of generators in \cite{Ha21} or the special instance here is optimal. It is quite possible that there is another construction with fewer generators.

\begin{question}
	Is the isomorphism problem for $5$-generated recursively presented groups $\Sigma^0_3$-complete? How about $2$-generated recursively presented groups?
\end{question}

\section{Classifying the ceers equivalent to a word problem of groups}

In this section we seek to understand the question of which ceers are equivalent to finitely presented groups. We show that this is a $\Sigma^0_3$-complete collection, so we cannot find a simpler description of this class than simply saying it is equivalent to the word problem of a finitely presented group.

Recall the following definitions: $(E_i)_{i \in \omega}$ is an effective uniform enumeration of the ceers; if $S$ is a set of indices of ceers, then $[S]$ is its closure under computable equivalence; and $S$ is proper on the infinite ceers if $[S]$  contains neither all nor none of the infinite ceers. Though $[S]$ is the set of indices of ceers, we abuse notation and write $R\in [S]$ (or $R\in S$) to mean that there is some $i\in [S]$ (respectively $i\in S$) such that $E_i = R$.

\begin{theorem}\label{Sigma3CompleteDegreeSets}
	If $S$ is $\Sigma^0_3$ and proper on the infinite ceers, then $[S]$ is $\Sigma^0_3$-complete.
\end{theorem}
\begin{proof}
	We begin by noting that if $S$ is $\Sigma^0_3$, then so is $[S]$. Now suppose that $S$ is proper on the infinite ceers, namely, $[S]$ does not contain or omit all infinite ceers. Thus we can fix $E \in S$ an infinite ceer, and $R$ an infinite ceer not in $[S]$. Let $Y(x,n)$ be a $\Pi^0_2$ relation so that $n\in S$ if and only if $\exists x Y(x,n)$.
	
	Let $U$ be a $\Sigma^0_3$ set which we will 1-reduce to $[S]$. Let $T(a,i)$ a $\Pi^0_2$ relation such that $i \in U$ if and only if $\exists a T(a,i)$.
	
	
	For each $i$, we will construct (uniformly in $i$) a ceer $Z_i$ such that if $i \in U$, then $Z_i \in [S]$, and if $i \notin U$, then $Z_i \notin [S]$. Thus the function that maps $i$ to an index for $Z_i$ will witness that $U \leq_1 [S]$. In building $Z_i$, we will attempt to meet the following requirements:
	\[ \mc{R}_a: \text{If $T(a,i)$, then $Z_i \equiv E$} \]
	and
	\[ \mc{S}_{\varphi,\psi,j,m}: \text{If $Y(m,j)$, then $\varphi,\psi$ do not witness an equivalence $Z_i \equiv E_j$}. \]
	Here, $\varphi$ and $\psi$ range over the potential computable reductions between ceers.
	
	We order the requirements in order type $\omega$, and attempt to meet them by a standard finite injury construction. However, unlike a finite injury construction, we may not meet all of the requirements because we may be blocked by one requirement that acts infinitely often.
	
	The requirement $\mc{R}_a$ is attempting to ensure that if $i \in U$, then $Z_i \in [S]$ by having $Z_i \equiv E \in [S]$. If the $\Pi^0_2$ relation $T(a,i)$ holds, then $a$ is a witness to the fact that $i \in U$. Thus, if $i \in U$ with least witness $a$, it will be $\mc{R}_a$ that acts infinitely many times (blocking each of the lower-priority $\mc{S}$ requirements).
	
	The $\mc S$-requirements are attempting to make $Z_i\notin [S]$. Recall that if $Y(m,j)$, then $m$ is a witness to $E_j \in S$, and thus we must ensure that $Z_i \not\equiv E_j$. Each individual requirement $\mc{S}_{\varphi,\psi,j,m}$ ensures that the particular reductions $\varphi$ and $\psi$ do not witness a computable equivalence.
	
	These requirements are of course in conflict. This conflict will be settled by priority. If $i \notin U$, then each $\mc{R}_a$ will eventually stop taking action (once we no longer guess that the $\Pi^0_2$ relation $T(a,i)$ holds) and so each $\mc{S}$ requirement will be met. If $i \in U$, with witness $a$, then  $\mc{R}_a$ will take action infinitely many times and so no lower priority $\mc{S}$ requirement will be met.
	
	For both $T$ and $Y$, we have a computable approximation so that we infinitely often guess that $T(a,i)$ holds if and only if it holds, and similarly for $Y$. At stage $s$, we say that an $\mc{R}_a$-requirement requires attention if the approximation at stage $s$ says $T(a,i)$ is true. For a $\mc{S}_{\varphi,\psi,j,m}$-requirement, let the parameter $N$ count the number of times it has acted since last initialization. We say that a $\mc{S}_{\varphi,\psi,j,m}$-requirement requires attention if the approximation at stage $s$ says that $Y(m,j)$ is true and for each $k,l<N$: $k \mathrel{Z_i} l \leftrightarrow \phi(k)\mathrel{E_j}\phi(l)$ and $k \mathrel{E_j} l \leftrightarrow \psi(k)\mathrel{Z_i}\psi(l)$ (in particular, we demand convergence for $\phi$ and $\psi$ on all inputs $<N$).
	
	At each stage, we let the highest priority requirement which requires attention act, and we reinitialize all lower-priority requirements. Each strategy may place a finite restraint on $Z_i$ meaning that lower priority requirements cannot cause $Z_i$-collapse between two numbers below this restraint.
	
	We next describe the strategy for each requirement:
	\begin{description}
		\item[$\mc{R}_a$-strategies:] When initialized, the strategy acts as follows: Let $n$ be the restraint placed by higher-priority requirements. Let $b_0,\ldots b_m$ be the least elements of each equivalence class containing a number $\leq n$. Let $a^s_0,\ldots, a^s_m$ be the least members of the first $m+1$ equivalence classes in $E_s$. If at a later stage we see two of these become $E$-equivalent, the requirement will re-initialize itself and begin again. We begin with the map $b_\ell\mapsto a_\ell^s$ for $\ell \leq k$, and we will attempt to build a reduction from $Z_i$ to $E$. We do this as follows: Whenever the requirement acts, it will have a partial map $\tau_s$. For any $j,k\in \dom(\tau_s)$, if $\tau_s(j) \mathrel{E} \tau_s(k)$, then we cause $j\mathrel{Z_i} k$. Also, we take the first $k\in \omega$ which is not yet $Z_i$-equivalent to a member of $\dom(\tau_s)$, and we define $\tau_s(k)$ to be the least element not currently $E_s$-equivalent to a member of $\ran(\tau_s)$. We also increase the restraint on lower-priority requirements to include this number.
	
		\item[$\mc{S}_{\varphi,\psi,j,m}$-strategies:] This strategy acts exactly as the $\mathcal{R}_a$-strategy except that instead of copying $E$ into $Z_i$, it acts by copying $R$ into $Z_i$.
	
	\end{description}
	
	\begin{lemma}\label{RsSucceed}
		If an $\mathcal{R}_a$-strategy is the first to act infinitely often, then $Z_i\equiv E$.
	\end{lemma}
	\begin{proof}
		Let $n$ be the restraint placed by higher-priority requirements when this strategy is last initialized. After some later stage, the elements $a^s_0,\ldots a^s_k$ are indeed the least $k$ elements which are least in their $E$-equivalence classes. From there, the strategy builds a computable function $\tau = \bigcup_s \tau_s$ with the property that every number is $Z_i$-equivalent to a member of the domain and every number is $E$-equivalent to a member of the range. Furthermore, as soon as a number enters the domain, it is restrained, so lower priority requirements will not collapse two numbers in the domain of $\tau_s$. Higher-priority requirements cannot ever act again, so we have that $j \mathrel{Z_i} k$ if and only if $\tau(j)\mathrel{E} \tau(k)$ for every $j,k\in \dom(\tau)$. 
		For any natural number $n$, define $f(n)=\tau(m)$ where $m$ is the first that we see is in $\dom(\tau)$ and is $Z_i$-equivalent to $n$. Then $f$ is a reduction of $Z_i$ to $E$. Similarly, for $n$ any natural number, let $g(n)$ be the first element $m$ so that we see $\tau(m)$ is $E$-equivalent to $n$. Then $g$ is a reduction of $E$ to $Z_i$.
	\end{proof}

	\begin{lemma}\label{SsSucceed}
		Let $\alpha$ be an $S_{\phi,\psi,j,m}$-strategy. Suppose every strategy which is higher-priority than $\alpha$ acts only finitely often. Then either $\neg Y(m,j)$ or $\phi,\psi$ do not witness $Z_i\equiv E_j$. In particular, $\alpha$ also acts only finitely often.
	\end{lemma}
	\begin{proof}
		Note that $\alpha$ acts infinitely often if and only if it requires attention infinitely often. That requires both $Y(m,j)$ and $\phi,\psi$ to witness $Z_i\equiv E_j$.
		Suppose towards a contradiction that both $Y(m,j)$ and $\phi,\psi$ witness $Z_i\equiv E_j$, so $\alpha$ acts infinitely often. Then the same proof as in Lemma \ref{RsSucceed} would show that $Z_i\equiv R$. But since $\phi$ and $\psi$ were assumed to witness $Z_i\equiv E_j$, that would imply that $R\equiv E_j$ with $E_j\in [S]$. This contradicts $R$ being non-equivalent to any element of $S$.
	\end{proof}
	
	It follows that if any requirement acts infinitely often, the first such must be an $R_a$-strategy. In this case, $T(a,i)$ so $i\in U$ and Lemma \ref{RsSucceed} shows that $Z_i\equiv E$, so $Z_i\in [S]$. If, on the other hand, each strategy acts only finitely often, then this means $\neg T(a,i)$ for each $a$, so $i\notin U$. Then Lemma \ref{SsSucceed} shows that $Z_i$ is not equivalent to any $E_j\in S$. That is, $Z_i\notin [S]$.
%
%
%
\end{proof}

By taking complements, we get the following corollary. Note that there is a slight difference in the statement, as if $S$ is $\Pi^0_3$, then it is not necessarily true that $[S]$ is $\Pi^0_3$.

\begin{corollary}
	If $S = [S]$ is $\Pi^0_3$ and proper on the infinite ceers, then $[S]$ is $\Pi^0_3$-complete.
\end{corollary}
\begin{proof}
	Let $T$ be the complement of $S = [S]$; then $T$ is $\Sigma^0_3$ and we also have that $T = [T]$. By the theorem, $T = [T]$ is $\Sigma^0_3$-complete, and thus $S = [S]$ is $\Pi^0_3$-complete.
\end{proof}

\begin{corollary}
	For any set $[S]$ which is proper on the infinite ceers, we cannot have $[S]<_m 0'''$ or $[S] <_1 0'''$.
\end{corollary}
\begin{proof}
	If $[S]$ were $\leq_m 0'''$, then in particular it would be $\Sigma^0_3$ and so Theorem \ref{Sigma3CompleteDegreeSets} would give that $[S]$ is $\Sigma^0_3$-complete, showing that $0'''\leq_m [S]$. (Note that $[S] = [[S]]$.) But since $[S]$ is an index set, we can use padding to get $0'''\leq_1 [S]$. Thus if $[S] \leq_m 0'''$, then $0''' \leq_1 [S]$.
\end{proof}

Finally, we have some examples:

\begin{example}
	Each of the following are $\Sigma^0_3$-complete:
	\begin{itemize}
		\item The indices of ceers equivalent to the word problem of a finitely presented group.
		\item The indices of ceers equivalent to the word problem of a finitely generated recursively presented group.
		\item The indices of ceers equivalent to the word problem of a recursively presented group.
	\end{itemize}
In each case, we can computably enumerate the collection of indices for word problems we are considering.
%
\end{example}

%

The following examples were already known:

\begin{example}[\cite{AndrewsSorbi}]
	Each of the following are $\Sigma^0_3$-complete:
	\begin{itemize}
		\item If $E_i$ has infinitely many classes, the set $[i]$ of ceers equivalent to $E_i$;
		\item If $E_i$ has infinitely many classes and is non-universal, the lower cone $\{j \mid E_j \leq E_i\}$;
		\item If $E_i$ has infinitely many classes, the upper cone $\{j \mid E_j \geq E_i\}$.
	\end{itemize}
\end{example}

\begin{example}
	The following are not $\Sigma^0_3$-complete, either because they are not $\Sigma^0_3$ or because they are not proper on the infinite ceers:
	\begin{enumerate}
		\item $[i]$ is $\Pi^0_2$-complete if $E_i$ has one class.
		
		\item $[i]$ is d-$\Sigma^0_2$-complete if $E_i$ has finitely many and more than one classes. 
		
		\item The set of indices for ceers with infinitely many classes is $\Pi^0_3$-complete.
		
		\item The collection of indices for non-universal ceers is a $\Pi^0_3$-complete set. Though this is proper on the infinite ceers, it is not generated by a $\Sigma^0_3$ set.
	\end{enumerate}
\end{example}


Next we aim to show that the hypothesis that $S$ (or $[S]$) is $\Sigma^0_3$ is optimal in Theorem \ref{Sigma3CompleteDegreeSets} (for sets $S$ proper on the infinite ceers). It is immediate from the last example that there are $\Pi^0_3$ sets $S=[S]$ which are thus not $\Sigma^0_3$-complete. We might conjecture for example that all sets $S = [S]$ are either $\Sigma^0_3$-hard or $\Pi^0_3$-hard. We will show that this is not the case. We first give a counterexample which is a $\Delta^0_4$ set $S=[S]$, and then we show that any such set is of the form $[T]$ for some $\Pi^0_3$ set $T$.

In the proof of the following theorem, for a ceer $R$, we will abuse notation and write $[R]$ to mean the set of indices for $R$. 

\begin{theorem}
	There is a $\Delta^0_4$ set $S = [S]$ which is proper on the infinite ceers and such that $S$ is neither $\Sigma^0_3$ nor $\Pi^0_3$-hard.
\end{theorem}
\begin{proof}
	We use $0'''$ to construct $S$. Our oracle knows whether or not $E_i \equiv E_j$ and whether a function is total. Fix $U$ a $\Sigma^0_3$-complete set. We write $U^0=U$ and $U^1$ for the complement of $U$.
	
	Begin at stage $0$ by fixing some infinite ceer $Q_0$ and make $Q_0 \in S$, and some infinite ceer $R_0$ and make $R_0 \notin S$. (When we put a ceer in or not in $S$, we put every equivalent ceer in or not in $S$.)
	
	At stage $\la e,k \ra$, $k \in \{0,1\}$, we deal with the potential reduction $\varphi_{e} : U^k \to [S]$, if $\varphi_e$ is total, as follows. There are finitely many $Q_0,\ldots,Q_\ell$ which we have committed to be in $S$, and finitely many $R_0,\ldots,R_m$ which we have committed to not be in $S$.
	
	We search for an $x$ such that one of the following is the case:
	\begin{enumerate}
		\item $\varphi_e(x) \notin [Q_0] \cup \cdots \cup [Q_\ell] \cup [R_0] \cup \cdots \cup [R_m]$;
		\item $x \notin U^k$ and $\varphi_e(x) \in [Q_0] \cup \cdots \cup [Q_\ell]$;
		\item $x \in U^k$ and $\varphi_e(x) \in [R_0] \cup \cdots \cup [R_m]$.
	\end{enumerate}
	Suppose towards a contradiction that there is no such $x$. Then for all $x$, $x\in U^k$ if and only if $\phi_e(x)\in  [Q_0] \cup \cdots \cup [Q_\ell] $. This is a $\Sigma^0_3$-condition. Similarly, $x\in U^{1-k}$ if and only if $\phi_e(x)\in  [R_0] \cup \cdots \cup [R_\ell] $. This is also a $\Sigma^0_3$-condition, so $U$ would be $\Delta^0_3$ contradicting its $\Sigma^0_3$-completeness. So $0'''$ finds such an $x$.
	
	In case (1), if $x \in U^k$, then we commit to $E_{\varphi_e(x)}$ not being in $S$. If $x \notin U^k$, then we commit to $E_{\varphi_e(x)}$ being in $S$. In case (2) or (3), we do not have to take any action.
	
	In all cases, choose the least $E_i$ which we have not yet committed to being in or out of $S$, and commit to it being in $S$. At each stage $\la e,k \ra$ we successfully ensured that $\phi_e$ is not a reduction of $U^k$ to $[S]$, so $[S]$ is neither $\Sigma^0_3$-hard nor $\Pi^0_3$-hard.
\end{proof}

Now we work towards replacing $S$ with a $\Pi^0_3$-set.

\begin{lemma}
	Let $S$ be a $\Sigma^0_2(X)$ set and $S = [S]$. Then there is a $\Pi^0_1(X)$ set $T$ with $[T] = [S] = S$.
\end{lemma}
\begin{proof}
	Fix a one-to-one and increasing padding function $f$ with computable range. Then $E_i = E_{f(i)}$.
	
	Let $g$ be a $\Sigma^0_2(X)$ approximation to $S$, i.e., $g$ is an $X$-computable function with
	\[ i \notin S \Longleftrightarrow \exists^\infty s \;\; g(i,s) = 1.\]
	We will define $T$ stage-by-stage. At stage $s$, for each $i < s$ not in the image of $f$, if $g(i,s) = 1$ then remove $E_i,E_{f(i)},\ldots,E_{f^{s}(i)}$ from $T$. (If $g(i,s) = 0$, then do nothing to these elements at this stage.)
	
	If $k \notin S$, then for any $j$ with $E_j \equiv E_k$, fix $i$ such that $j = f^t(i)$ and $i$ is not in the image of $f$. At some stage $t' > t$, we will have $g(i,t') = 1$, and $E_j = E_{f^t(i)}$ will be removed from $T$. Thus $j \notin T$, and since $j$ was arbitrary, $k \notin [T]$.
	
	If $k \in S$, fix $i$ such that $E_i \equiv E_k$ and $i$ is not in the range of $f$. Then for some $t$ sufficiently large, $E_{f^t(i)}$ will never be removed from $T$. Thus $f^t(i) \in T$ and so $k \in [T]$.
\end{proof}

\begin{corollary}
	There is a $\Pi^0_3$ set $T$ which is proper on the infinite ceers such that $[T]$ is neither $\Sigma^0_3$ nor $\Pi^0_3$-hard.
\end{corollary}
\begin{proof}
	Let $S$ be the $\Delta^0_4$ set from the previous theorem. The previous Lemma gives a $\Pi^0_3$ set $T$ with $[T] = [S] = S$.
\end{proof}

\bibliographystyle{plain}
\bibliography{Zzyzx}

\begin{thebibliography}{10}

\bibitem{Ad55}
S.~I. Adyan.
\newblock Algorithmic unsolvability of problems of recognition of certain
  properties of groups.
\newblock {\em Dokl. Akad. Nauk SSSR (N.S.)}, 103:533--535, 1955.

\bibitem{ASM1}
Uri Andrews and Luca San~Mauro.
\newblock Analogs of the countable borel equivalence relations in the setting
  of computable reducibility.
\newblock {\em submitted}.

\bibitem{AndrewsSorbi}
Uri Andrews and Andrea Sorbi.
\newblock The complexity of index sets of classes of computably enumerable
  equivalence relations.
\newblock {\em J. Symb. Log.}, 81(4):1375--1395, 2016.

\bibitem{Bo67}
L.~A. Bokut.
\newblock On a property of the {B}oone groups. {II}.
\newblock {\em Algebra i Logika Sem.}, 6(1):15--24, 1967.

\bibitem{Bo57}
William~W. Boone.
\newblock Certain simple, unsolvable problems of group theory. {V}, {VI}.
\newblock {\em Indag. Math.}, 19:22--27, 227--232, 1957.
\newblock Nederl. Akad. Wetensch. Proc. Ser. A {\bf 60}.

\bibitem{Bo66}
William~W. Boone.
\newblock Word problems and recursively enumerable degrees of unsolvability.
  {A} sequel on finitely presented groups.
\newblock {\em Ann. of Math. (2)}, 84:49--84, 1966.

\bibitem{Cl64}
C.~R.~J. Clapham.
\newblock Finitely presented groups with word problems of arbitrary degrees of
  insolubility.
\newblock {\em Proc. London Math. Soc. (3)}, 14:633--676, 1964.

\bibitem{De11}
M.~Dehn.
\newblock \"{U}ber unendliche diskontinuierliche {G}ruppen.
\newblock {\em Math. Ann.}, 71(1):116--144, 1911.

\bibitem{De23}
Valentino Delle~Rose, Luca San~Mauro, and Andrea Sorbi.
\newblock Classifying word problems of finitely generated algebras via
  computable reducibility.
\newblock {\em Internat. J. Algebra Comput.}, 33(4):751--768, 2023.

\bibitem{Fr62}
A.~A. Fridman.
\newblock Degrees of unsolvability of the word problem for finitely presented
  groups.
\newblock {\em Dokl. Akad. Nauk SSSR}, 147:805--808, 1962.

\bibitem{Ha21}
Matthew Harrison-Trainor and Meng-Che Ho.
\newblock Finitely generated groups are universal among finitely generated
  structures.
\newblock {\em Ann. Pure Appl. Logic}, 172(1):Paper No. 102855, 21, 2021.

\bibitem{CharlesFMillerIII}
Charles~F. Miller, III.
\newblock {\em On group-theoretic decision problems and their classification},
  volume No. 68 of {\em Annals of Mathematics Studies}.
\newblock Princeton University Press, Princeton, NJ; University of Tokyo Press,
  Tokyo, 1971.

\bibitem{No55}
P.~S. Novikov.
\newblock {\em Ob algoritmi\v{c}esko\u{\i} nerazre\v{s}imosti problemy
  to\v{z}destva slov v teorii grupp}.
\newblock Izdat. Akad. Nauk SSSR, Moscow, 1955.
\newblock Trudy Mat. Inst. Steklov. no. 44.

\end{thebibliography}

\end{document}